\documentclass[reqno]{amsart}
\usepackage{amsmath,amsbsy,amsfonts}
\usepackage{color}
\usepackage[colorlinks=true,urlcolor=blue,
citecolor=red,linkcolor=blue,linktocpage,pdfpagelabels,
bookmarksnumbered,bookmarksopen]{hyperref}

\usepackage[utf8]{inputenc}
\usepackage{graphicx}
\usepackage{amsmath,latexsym,amssymb}
\newtheorem{theorem}{Theorem}
\newtheorem{corollary}[theorem]{Corollary}
\newtheorem{definition}{Definition}[section]
\newtheorem{lemma}[theorem]{Lemma}

\newtheorem{remark}{Remark}[section]

\newcommand{\dd}{\mathrm{d}}

\numberwithin{equation}{section}

\begin{document}
\title[Fractional Schr\"{o}dinger equations with Neumann condition]{Existence of nonnegative solutions for fractional Schr\"{o}dinger equations with Neumann condition}	

\author{H. Bueno}
\address{Departmento de Matem\'atica, Universidade Federal de  Minas Gerais, 31270-901 - Belo Horizonte - MG, Brazil}
\email{hamilton@mat.ufmg.br}
		
\author{Aldo H. S. Medeiros}
\address{Departamento de Matem\'{a}tica,
Universidade Federal de Viçosa, 36570-900 - Vi\c{c}osa - MG, Brazil.}
\email{aldo.medeiros@ufv.br}
	

\begin{abstract}
In this paper we study a Neumann problem for the fractional Laplacian, namely
\begin{equation}\left\{
\begin{array}{rcll}
\varepsilon^{2s}(- \Delta)^{s}u + u &=& f(u) \ \ &\mbox{in} \ \ \Omega  \\ 
\mathcal{N}_{s}u &=& 0 , \,\, &\text{in} \,\, \mathbb{R}^{N}\backslash \Omega 
\end{array}\right.
\end{equation}
where $\Omega \subset \mathbb{R}^{N}$ is a smooth bounded domain, $N>2s$, $s \in (0,1)$, $\varepsilon > 0$ is a parameter and $\mathcal{N}_{s}$ is the nonlocal normal derivative introduced by Dipierro, Ros-Oton, and Valdinoci. We establish the existence of a nonnegative, non-constant small energy solution $u_{\varepsilon}$, and we use the Moser-Nash iteration procedure to show that $u_{\varepsilon} \in L^{\infty}(\Omega)$. 
\end{abstract}
	\maketitle
\section{Introduction}
In this paper, we study a Neumann elliptic problem for an equation driven by the fractional Laplacian. More precisely, we consider the problem
\begin{equation}\label{P}\left\{
\begin{array}{rcll}
\varepsilon^{2s}(- \Delta)^{s}u + u &=& f(u) &\mbox{in} \ \ \Omega,  \\ 
\mathcal{N}_{s}u &=& 0 &\text{in}\ \ \mathbb{R}^{N}\setminus \Omega,
\end{array}\right.
\end{equation}
where $\Omega \subset \mathbb{R}^{N}$ is a smooth bounded domain, $N>2s$, $s \in (0,1)$, $\varepsilon > 0$ is a parameter and $\mathcal{N}_{s}u$ is the nonlocal normal derivative defined by
\begin{align}\label{neumman}
\mathcal{N}_{s}u(x) = C_{N,s}\int_{\Omega} \frac{u(x) -u(y)}{\vert x-y \vert^{N+2s}} \dd y, \quad \quad  x \in \mathbb{R}^{N}\backslash \Omega.
\end{align}
where $C_{N,s}$ is the normalization constant of the fractional Laplacian, defined for smooth functions by
\begin{align*}
(-\Delta)^{s}\phi(x) = C_{N,s}\int_{\mathbb{R}^{N}} \frac{\phi(x) -\phi(y)}{\vert x-y \vert^{N+2s}} \dd y,
\end{align*}
with both integrals being understood in the principle value sense. One advantage of the present approach is that the integration by parts formulas
\[\int_\Omega \Delta u=\int_{\partial\Omega}\partial_\nu u\quad\textrm{and}\quad \int_\Omega \nabla u\cdot \nabla v=\int_\Omega v (-\Delta u)+\int_\Omega v\partial_\nu u\]
are substituted, respectively, by 
\[\int_\Omega (-\Delta)^s u=-\int_{\Omega^c} \mathcal{N}_{s}u(x)\]
and 
\[\frac{C_{N,s}}{2}\iint_{\mathbb{R}^{2N}\setminus (\Omega^c)^2}\frac{(u(x)-u(y))(v(x)-v(y))}{\vert x-y\vert^{N+2s}} \dd x \dd y =\int_\Omega v(-\Delta)^s u+\int_{\Omega^c}v\mathcal{N}_s u,\]
where $\Omega^c=\mathbb{R}^N\setminus\Omega$ and $\left(\Omega\right)^{2} = \Omega \times \Omega$. For further details on the fractional Neumann derivative $\mathcal{N}_{s}u$, see Dipierro, Ros-Oton, and Valdinoci \cite{MR3651008}, where this concept was introduced.

This type of boundary problem for the fractional Laplacian has a probabilistic interpretation: if a particle has gone to $x \in \mathbb{R}^{N}\setminus \overline{\Omega}$, then it may come back to any point $y \in  \Omega$, the probability of jumping from $x$ to $y$ being proportional to $\vert x-y\vert^{-N-2s}$. So, it generalizes the classical Neumann conditions for elliptic (or parabolic) differential equations since, as $s \to 1$, then $\mathcal{N}_su = 0$ turns into the classical Neumann condition. For more details, see \cite{MR3651008} and also \cite{du,duq}. 

Du et al. introduced volume constraints for a general class of nonlocal diffusion problems on a bounded domain in $\mathbb{R}^N$ via a nonlocal vector calculus. If we rewrite \eqref{neumman} using that vector calculus, then a modified version of $\mathcal{N}_su=0$ can be considered as a particular case of the volume constraints defined by them.

Neumann problems for the fractional Laplacian and other nonlocal operators were introduced in \cite{MR3217703, MR3145761, MR2300660, MR2358337}. All these generalizations to  nonlocal operators recover the classical Neumann problem as a limit case, and most also have clear probabilistic interpretations. In Dipierro et al. \cite[Section 7]{MR3651008}, the authors compared all these models with the one considered here. 

The case $f(t) = \vert t \vert^{p-1}t$ with $1<p<\frac{N+2s}{N-2s}$, which is known as the  singularly perturbed Neumann problem,  was studied by Guoyuan Chen in \cite{Chen}. The author established the existence of non-negative small energy solutions and investigated their integrability in $\mathbb{R}^{N}$.

When $s=1$, the problem \eqref{P} reduces to the Laplacian case, considered in the classical paper by Lin, Ni, and Takagi \cite{MR929196}, which studies the existence of solutions to the semilinear Neumann boundary problem
\begin{equation}\label{P_2}\left\{
\begin{array}{rcll}
\varepsilon^{2}(- \Delta)u + u &=& g(u) \ \ &\mbox{in} \ \ \Omega  \\ 
\frac{\partial u}{\partial \nu} &=& 0 , \,\, &\text{on} \,\, \partial\Omega 
\end{array}\right.
\end{equation}
where $\nu$ denotes the outer normal to $\partial \Omega$ and $g(t)$ is a suitable nonnegative nonlinearity on $\mathbb{R}$ vanishing for $t\leq 0$, growing superlinearly at infinity. It was shown that, if $\varepsilon$ is small enough, there exists a positive smooth solution $u_{\varepsilon}$ that satisfies $J_{\varepsilon}(u_{\varepsilon}) \leq C\varepsilon^{\frac{N}{2}}$, where $C$ is a positive constant independent of $\varepsilon$ and $J_{\varepsilon}$ is the energy functional of problem \eqref{P_2}.

Stinga-Volzone \cite{MR3385190} extended the results in \cite{MR929196} to the square root of the Laplacian, obtaining similar results. More precisely, they considered problem
\begin{equation}\label{P_3}\left\{
\begin{array}{rcll}
\varepsilon(- \Delta)^{\frac{1}{2}}u + u &=& g(u) \ \ &\mbox{in} \ \ \Omega  \\ 
\frac{\partial u}{\partial \nu} &=& 0 , \,\, &\text{on} \,\, \partial\Omega 
\end{array}\right.
\end{equation} 
for the nonlinearity
\begin{equation}\label{gStinga}
g(t) = \left\{
\begin{array}{rcll}
t^{p} & \mbox{if} \ \ t \geq 0,  \\ 
0 & \text{if} \ \ t \leq 0, 
\end{array}\right.
\end{equation}
with $1<p<\frac{N+1}{N-1}$. 

Recently, Haige Ni, Aliang Xia, and Xiogjun Zheng \cite{XIA} studied the problem 
\begin{equation}\label{eqq} 
\left\{
\begin{array}{rcll}
\varepsilon^{2s}(- \Delta)^{s}u + u &=& g(u) \ \ &\mbox{in} \ \ \Omega  \\ 
\frac{\partial u}{\partial \nu} &=& 0 , \,\, &\text{on} \,\, \partial\Omega \\
u &>& 0 \,\, \mbox{in} \,\, \Omega 
\end{array}\right.
\end{equation}
where $g$ satisfies \eqref{gStinga} and $s\in (0,s_0)$, with $s_0\geq \frac{1}{2}$. The authors used the extension technique to obtain the existence of nonnegative solutions for $\varepsilon$ small enough and $L^{\infty}$-estimates to show that they are bounded. In their paper, they considered the spectral fractional Laplacian, which differs from its integral form, see \cite{M98780, MR3233760}. By applying the Mountain Pass Theorem of Ambrosetti and Rabinowitz, they proved the existence of nonconstant solutions of \eqref{eqq} provided $\varepsilon$ is small. They also studied regularity and the Harnack inequality in the same paper. 

Here, we study problem \eqref{P} considering the normal derivative defined by Dipierro, Ros-Oton, and Valdinoci in \cite{MR3651008}. We suppose that the continuous  nonlinearity $f$ satisfies the following conditions.
\begin{enumerate}
\item [$(f_1)$] $f(t) = 0$ for $t<0$, and $f(t) > 0$ for $t>0$;
\item[$(f_2)$] $\displaystyle\lim_{t \to 0^{+}} \frac{f(t)}{t} = 0$, and $\displaystyle\lim_{t \to \infty}\frac{f(t)}{t^{p-1}} = 0$ for some $2 < p < \frac{2N}{N-2s}=2^*_s$;
\item[$(f_3)$]  $\displaystyle\lim_{t \to \infty}\frac{f(t)}{t} = + \infty$;
\item[$(f_4)$] There exist  $\theta > 2$ and $a_3 \geq 0$, such that 
\begin{align*}
0< \theta F(t) \leq t f(t), \,\,\, \forall t \geq a_3,
\end{align*}
where $F(t)$ denotes the primitive of $f$.
\item[$(f_5)$] $\alpha := \inf\left\{ \frac{t^2}{2} - F(t); \,\, t \in \,\,\,\text{Fix}(f) \right\} > 0,$ where Fix$(f) = \left\{t >0; \,\,\, f(t) = t\right\}.$
\end{enumerate}
Condition $(f_5)$ permits us to discard constant solutions.
\begin{remark}\label{obs1}	It follows from \textup{($f_1$)} and \textup{($f_2$)} that, for any fixed $\eta>0$ (or any fixed $C_\eta>0$), there exists a constant $C_\eta$ (respectively, $\eta>0$) such that
\begin{equation}\label{boundf}|f(t)|\leq\eta t+C_\eta t^{p-1},\quad\forall\ t\geq 0\end{equation}
and analogously, denoting $F(t) = \displaystyle\int_{0}^{t} f(s) \dd s$ we have
\begin{equation}\label{boundF}|F(t)|\leq\eta t^2+C_\eta t^{p}\leq C(t^2+t^p),\quad\forall\ t\geq 0\end{equation}
for any $2<p<2^*_s = \displaystyle\frac{2N}{N-2s}$.
\end{remark}\goodbreak

Our first result is the following.
\begin{theorem}\label{Exis}
Assume $(f_1)$-$(f_5)$. Then, for $\varepsilon$ sufficiently small, there exists a non-constant, nonnegative solution of \eqref{P} satisfying 
\[
I_{\varepsilon}(u_\varepsilon) \leq C\varepsilon^{N}
\]  
where $C>0$ depends only on $\Omega$ and $f$.
\end{theorem}

We use the Mountain Pass Theorem of Ambrosetti and Rabinowitz to prove this result, see \cite{MR3002745, MR1400007}. The main difficulties arise from the degeneracy of the operator and also from the geometry of the problem. 

We also prove the following result.
\begin{theorem}\label{Limita}
Suppose $0<s<1$, $(f_1)$-$(f_3)$ holds. If $u_\varepsilon$ is a solution  to problem \eqref{P} with $\varepsilon > 0$ small enough, then $u_{\varepsilon} \in  L^{\infty}(\Omega)$.
\end{theorem}

We prove Theorem \ref{Limita} by using Moser-Nash's iteration method (see \cite{inbook}), which has been used to study uniform bounds for fractional elliptic problems, see  \cite{M98210, Alves2016, Ambrosio1, M83927, Ambrosio2, Jianfu}.
\section{Variational formulation}
Problem \eqref{P} has a variational structure. More precisely, consider
\begin{equation}\label{PI}
\langle u, v\rangle_{\varepsilon,s} := \frac{C_{N,s}\varepsilon^{2s}}{2} \iint_{\mathbb{R}^{2N}\backslash \left(\Omega^{c}\right)^{2}} \frac{(u(x)-u(y))(v(x)-v(y))}{\vert x-y\vert^{N+2s}} \dd x \dd y + \int_{\Omega} uv \dd x
\end{equation}
where $\Omega^{c} = \mathbb{R}^{N}\backslash \Omega$ and $\left(\Omega\right)^{2} = \Omega \times \Omega$. The space
\[
H_{\varepsilon}^{s}(\Omega) := \left\{u: \mathbb{R}^{N} \to \mathbb{R} \,\,\, \text{measurable and} \,\, \langle u,u\rangle_{\varepsilon,s} < \infty \right\}
\]
is a Hilbert space with the norm $
\Vert u \Vert_{H_{\varepsilon}^{s}(\Omega)} = \langle u,u\rangle^{1/2}_{\varepsilon,s}$, see \cite{MR3651008} for details.

\begin{remark}\label{obs2}
Note that constant functions are contained in $H_{\varepsilon}^{s}(\Omega)$, see \cite{Chen}. Moreover, for all $u \in H_{\varepsilon}^{s}(\Omega)$, we have that 
$u|_{\Omega} \in H^{s}(\Omega)$. Using the compact embedding $H^{s}(\Omega) \hookrightarrow L^{q}(\Omega)$ for $q \in \left(1, \frac{2N}{N-2s}\right)$, we conclude that the embedding
\[
H_{\varepsilon}^{s}(\Omega) \hookrightarrow L^{q}(\Omega),\,\,\, \text{for all} \,\,\, 1 < q < \frac{2N}{N-2s}.
\]
is compact. So, if $(u_n)$ is bounded sequence in $H_{\varepsilon}^{s}(\Omega)$, then $u_{n}|_{\Omega}$ has a convergence subsequence in $L^q(\Omega)$.
\end{remark}
More precisely, considering the Sobolev constant,
\begin{equation}\label{Sob}
S = \displaystyle\inf_{u \in H^{s}_\varepsilon(\Omega), u \neq 0} \displaystyle\frac{\left(\frac{C_{N,s}}{2}\displaystyle\int_{\Omega}\displaystyle\int_{\Omega} \displaystyle\frac{\vert u(x) - u(y) \vert^{2}}{\vert x- y \vert^{N+2s}} \dd x \dd y\right)^{\frac{1}{2}}}{\left(\displaystyle\int_{\Omega} \vert u(x) \vert^{2_s^{*}} \dd x \right)^{\frac{1}{2_s^{*}}}}
\end{equation}we have the following Sobolev inequality:
\begin{lemma}\label{DS}
Let $\Omega \subset \mathbb{R}^{N}$ bounded and  $\varepsilon>0$. Then 
\begin{equation}\label{SI}
\left(\int_{\Omega} \vert u \vert^{2^{*}_{s}} \dd x\right)^{\frac{2}{2^{*}_{s}}} \leq S^{2} \varepsilon^{-2s} \Vert u \Vert_{H^{s}_{\varepsilon}(\Omega)}^{2}, \quad \forall u \in H^{s}_{\varepsilon}(\Omega).
\end{equation}
where $S$ is the Sobolev constant defined in \eqref{Sob}. 
\end{lemma}
\begin{proof}
For any fixed $u \in H^{s}_{\varepsilon}(\Omega)$, consider the function $v_{\varepsilon}(x) = u(\varepsilon x)$ defined in $\Omega_{\varepsilon} = \left\{x \in \mathbb{R}^{N}\,:\, \varepsilon x \in \Omega \right\}$. It follows from the Sobolev inequality that
\begin{align*}
\Vert u \Vert_{H^{s}_{\varepsilon}(\Omega)}^{2} &= \frac{C_{N,s}\varepsilon^{2s}}{2}\int_{\Omega}\int_{\Omega} \frac{\vert u(x) - u(y) \vert^{2}}{\vert x-y \vert^{N+2s}} \dd x \dd y + \int_{\Omega} \vert u(x) \vert^{2} \dd x \\
&=  \frac{C_{N,s}\varepsilon^{N+2s}}{2}\int_{\Omega_{\varepsilon}}\int_{\Omega} \frac{\vert u(x) - u(\varepsilon y) \vert^{2}}{\vert x-\varepsilon y \vert^{N+2s}} \dd x \dd y + \varepsilon^{N}\int_{\Omega_{\varepsilon}}\vert u(\varepsilon x) \vert^{2} \dd x \\
&= \frac{C_{N,s}\varepsilon^{N+2s}}{2}\varepsilon^{N}\int_{\Omega_{\varepsilon}}\int_{\Omega_{\varepsilon}} \frac{\vert u(\varepsilon x) - u(\varepsilon y) \vert^{2}}{\vert \varepsilon x-\varepsilon y \vert^{N+2s}} \dd x \dd y + \varepsilon^{N}\int_{\Omega_{\varepsilon}}\vert u(\varepsilon x) \vert^{2} \dd x \\
&= \varepsilon^{N}\left[ \frac{C_{N,s}\varepsilon^{N+2s}}{2}\int_{\Omega_{\varepsilon}}\int_{\Omega_{\varepsilon}} \frac{\vert v_{\varepsilon}(x) - v_{\varepsilon}(y) \vert^{2}}{\vert \varepsilon x-\varepsilon y \vert^{N+2s}} \dd x \dd y + \int_{\Omega_{\varepsilon}} \vert v_{\varepsilon}(x) \vert^{2} \dd x\right] \\
&\geq \frac{\varepsilon^{N}}{S^{2}} \left(\int_{\Omega_{\varepsilon}} \vert v_{\varepsilon}(x) \vert^{2_{s}^{*}} \dd x \right)^{\frac{2}{2_{s}^{*}}} = \frac{\varepsilon^{N\left(1-\frac{2}{2_{s}^{*}}\right)}}{S^{2}} \left(\int_{\Omega} \vert u(x) \vert^{2_{s}^{*}} \dd x \right)^{\frac{2}{2_{s}^{*}}}.
\end{align*}
Since 
\[N\left(1-\frac{2}{2_{s}^{*}}\right) = N\left(1 -\frac{N-2s}{N}\right) = 2s,\]
we obtain
\[\left(\int_{\Omega} \vert u(x) \vert^{2_{s}^{*}} \dd x \right)^{\frac{2}{2_{s}^{*}}} \leq S^{2} \varepsilon^{-2s} \Vert u \Vert_{H^{s}_{\varepsilon}(\Omega)}^{2}.
\vspace*{-.5cm}\]

$\hfill\Box$\end{proof}\vspace*{.3cm}

\begin{definition}\label{DEF1}
We say that $u \in H_{\varepsilon}^{s}(\Omega)$ is a weak solution of \eqref{P} if
\[
\frac{C_{N,s}\varepsilon^{2s}}{2}\iint_{\mathbb{R}^{2N}\backslash \left(\Omega^{c}\right)^{2}} \frac{(u(x)-u(y))(v(x)-v(y))}{\vert x-y \vert^{N+2s}} \dd x\dd y + \int_{\Omega} uv \dd x - \int_{\Omega}f(u) v \dd x = 0
\]
for all $v \in H_{\varepsilon}^{s}(\Omega)$.
\end{definition}

For all $u,v \in C^{2}(\mathbb{R}^{N}) \cap H_{\varepsilon}^{s}(\Omega)$, it follows from a direct computation that
\[\frac{C_{N,s}}{2} \iint_{\mathbb{R}^{2N}\backslash \left(\Omega^{c}\right)^{2}} \frac{(u(x)-u(y))(v(x)-v(y))}{\vert x-y \vert^{N+2s}} \dd x\dd y = \int_{\Omega} v (-\Delta)^{s}u \dd x + \int_{\Omega^c} v \mathcal{N}_{s} u \dd x,
\]
what yields
\[
\int_{\Omega} \left(\varepsilon^{2s} (-\Delta)^{s}u + u - f(u)\right) v \dd x + \varepsilon^{2s}\int_{\Omega^{c}} v\mathcal{N}_{s}u \dd x = 0.
\]
Thus, for  $x\in\mathbb{R}^N\setminus\Omega$,
\[\int_{\Omega^c}v(x)\left(C_{N,s}\int_{\Omega} \frac{u(x) -u(y)}{\vert x-y \vert^{N+2s}} \dd y\right)\dd x=0,\]
meaning that we have, weakly, $\mathcal{N}_{s}u = 0$. \goodbreak

Let us define, for all $u \in H_{\varepsilon}^{s}(\Omega)$,
\[I_{\varepsilon}(u) = \frac{C_{N,s}\varepsilon^{2s}}{4} \iint_{\mathbb{R}^{2N}\setminus \left(\Omega^c \right)^{2}} \frac{\vert u(x) - u(y) \vert^{2}}{\vert x-y \vert^{N+2s}} \dd x \dd y +\frac{1}{2} \int_{\Omega} \vert u \vert^{2} \dd x- \int_{\Omega}F(u) \dd x.
\]

As an easy consequence of Remark \eqref{obs1} and of the above discussion, we have that the functional $I_{\varepsilon}$ is well-defined and $I_{\varepsilon} \in C^{1}(H_{\varepsilon}^{s}(\Omega)), \mathbb{R})$.

The derivative of the functional $I_{\varepsilon}$ is given by
\begin{align*}
I_{\varepsilon}'(u)\cdot v &= \frac{C_{N,s}\varepsilon^{2s}}{2}\iint_{\mathbb{R}^{2N}\backslash \left(\Omega^c \right)^{2}}\frac{(u(x)-u(y))(v(x)-u(y))}{\vert x-y \vert^{N+2s}} \dd x \dd y + \int_{\Omega} u v \dd x\\&\quad - \int_{\Omega}f(u) v\dd y.
\end{align*}

Therefore, critical points of $I_{\varepsilon}$ are weak solutions of \eqref{P}.
\section{Proof of Theorem \ref{Exis}} \mbox{}
With arguments similar to that of Lin, Ni, and Takagi \cite{MR929196}, we prove Theorem \ref{Exis}, which is a consequence of the following lemmas.
\begin{lemma}\label{gpm}
There exist $\rho,\delta>0$ such that $I_{\varepsilon}|_S\geq \delta>0$ for all $u\in S$, where
\[S=\left\{u\in H_{\varepsilon}^s(\Omega)\,:\, \Vert u\Vert_{H_{\varepsilon}^{s}} =\rho\right\}. \]
\end{lemma}

\begin{proof}
Maintaining the notation of Remark \eqref{obs1}, the Sobolev embedding yields
\begin{align*}
I_{\varepsilon}(u) &=\frac{C_{N,s}\varepsilon^{2s}}{4} \iint_{\mathbb{R}^{2N}\setminus \left(\Omega \right)^{2}} \frac{\vert u(x) - u(y) \vert^{2}}{\vert x-y \vert^{N+2s}} \dd x \dd y + \frac{1}{2}\int_{\Omega} \vert u \vert^{2} \dd x- \int_{\Omega}F(u) \dd x \\
& \geq  \frac{1}{2}\Vert u \Vert_{H_{\varepsilon}^{2}(\Omega)}^{2} - \eta \int_{\Omega} \vert u\vert^2 \dd y -C_{\eta} \int_{\Omega} \vert u \vert^{p} \dd y = \frac{1}{2}\Vert u \Vert_{H_{\varepsilon}^{2}(\Omega)}^{2} - \eta \vert u \vert_{2}^{2} - C_{\eta} \vert u \vert_{p}^{p} \\
&\geq \left(\frac{1}{2} - \eta \right) \Vert u \Vert_{H_{\varepsilon}^{2}(\Omega)}^{2} - S^2\varepsilon^{-2s} \Vert u \Vert_{H_{\varepsilon}^{2}(\Omega)}^{p} .
\end{align*}
	
Taking $0 < \eta < \frac{1}{2}$, denote by $a = \displaystyle\frac{1}{2} - \eta$ and $A>0 = S^2\varepsilon^{-2s}$. So we obtain 
\begin{align*}
I_{\varepsilon}(u) &\geq a\Vert u \Vert_{H_{\varepsilon}^{2}(\Omega)}^{2} - A\Vert u \Vert_{H_{\varepsilon}^{2}(\Omega)}^{p}, \,\,\, \text{for all} \,\,\, u \in H_{\varepsilon}^{s}{(\Omega)}\\
&\geq \Vert u \Vert_{H^s_\varepsilon(\Omega)}^2 \bigg( a - A \Vert u \Vert_{H_\varepsilon^s(\Omega)}^{p-2} \bigg).
\end{align*}

Since $p \in \left(2, \frac{2N}{N-2s}\right)$, for $\rho \leq \left(\displaystyle\frac{a}{A}\right)^{\frac{1}{p-2}}$ we have 
\[
I_{\varepsilon}(u) \geq \rho^2(a-A\rho^{p-2}) > 0, \ \ \text{for all} \ \ \Vert u \Vert_{H_\varepsilon^{s}(\Omega)} = \rho.
\] 

\vspace*{-.5cm}$\hfill\Box$\end{proof}
\begin{lemma}\label{PScond}
If $(f_1)$-$(f_4)$ hold, then $I_{\varepsilon}$ satisfies the Palais-Smale condition.
\end{lemma}
\begin{proof}
Let $(u_n)$ be a (PS)-sequence for $I_{\varepsilon}$ in $H_{\varepsilon}^{s}(\Omega)$. Thus,
\[
I_{\varepsilon}(u_n) \leq k_0 \qquad \text{and} \qquad I_{\varepsilon}'(u_n) \to 0.
\]	
	
Consequently, there exists $n_0 \in \mathbb{N}$ such that
\[
\left\vert \Vert u_n \Vert_{H_{\varepsilon}^{s}} - \int_{\Omega} f(u_n)u_n \dd x \right\vert = \vert I_{\varepsilon}'(u_n)\cdot u_n \vert \leq \Vert u_n \Vert_{H_{\varepsilon}^{s}}^{2} ,\,\,\,\, \forall n \geq n_0.
\]	
It follows from condition $(f_4)$ the existence of $\theta > 2$ and $a_3>0$ such that $0<\theta F(t) \leq tf(t), \ \ \forall t \geq a_3$.

So, we obtain
\begin{align*}
\frac{1}{2} \Vert u_n \Vert_{H_{\varepsilon}^{s}}^{2} - k_0 &\leq \int_{\Omega} F(u_n) \dd x \\
&\leq \frac{1}{\theta} \int_{\{x \in \Omega ; u_n \geq a_3\}}f(u_n)u_n \dd x + \int_{\{x \in \Omega ; u_n \leq a_3\}}F(u_n) \dd x  \\
&\leq \frac{1}{\theta} \left( \Vert u_n \Vert_{H_{\varepsilon}^{2}(\Omega)} + \Vert u_n \Vert_{H_{\varepsilon}^{2}(\Omega)}^{2} \right) + \int_{\{x \in \Omega ; u_n \leq a_3\}}F(u_n) \dd x  \\
&\leq \frac{1}{\theta} \Vert u_n \Vert_{H_{\varepsilon}^{2}(\Omega)} + \frac{1}{\theta} \Vert u_n \Vert_{H_{\varepsilon}^{2}(\Omega)}^{2} + A_1,
\end{align*}
where $A_1 = \vert \Omega \vert \left(\displaystyle\max_{0\leq t \leq a_3}F(t) \right) < \infty$.

Therefore, for all $n\geq n_0$,
\[
\left(\frac{1}{2} - \frac{1}{\theta} \right) \Vert u_n \Vert_{H_{\varepsilon}^{2}(\Omega)}^{2} \leq \Vert u_n \Vert_{H_{\varepsilon}^{2}(\Omega)} + k_1.
\]

Since $\theta > 2$, it follows that $(u_n)$ is bounded in $H_{\varepsilon}^{s}(\Omega)$.

Thus, for a subsequence 
\[
u_n \rightharpoonup u \,\,\, \text{in} \,\, H_{\varepsilon}^{s}(\Omega) \,\,\, \text{and} \,\,\, u_n \to u \,\,\, \text{in} \,\,\, L^{q}(\Omega),
\]
for all $q \in \left(1,\frac{2N}{N-2s} \right)$.

Condition $(f_2)$ allows us to conclude that (for a subsequence) we have
\begin{equation}\label{Des1}
\int_{\Omega} (f(u_n) - f(u))(u_n-u) \dd x \to 0 \,\,\,\, \text{as} \,\,\,\, n \to \infty.
\end{equation}

Combining \eqref{Des1} with the identity
\[
(I_{\varepsilon}'(u_n) - I_{\varepsilon}'(u))\cdot (u_n-u) = \Vert u_n - u \Vert_{H_{\varepsilon}^{2}(\Omega)}^{2} + \int_{\Omega}(f(u_n) - f(u))(u_n-u) \dd x,
\]
it follows that
\[
\lim_{n \to \infty}\Vert u_n - u \Vert_{H_{\varepsilon}^{2}(\Omega)}^{2} = \lim_{\to}(I_{\varepsilon}'(u_n) - I_{\varepsilon}'(u))\cdot (u_n-u) = 0,
\]
that is, $u_n \to u$ in $H_{\varepsilon}^{s}(\Omega)$.
$\hfill\Box$\end{proof}\vspace*{.3cm}

From now on, without loss of generality, we assume that $0 \in \Omega$. For any $\varepsilon > 0$ such that $B_{\varepsilon}(0) \subset \Omega$, following Lin, Ni, and Takagi \cite{MR929196} we define
\[
\phi_{\varepsilon}(x) = \left\{
\begin{array}{cl}
\varepsilon^{-N}\left(1- \displaystyle\frac{\vert x\vert}{\varepsilon}\right) \ \ &\mbox{if} \ \ \vert x\vert \leq \varepsilon,  \\ 
0  &\text{if} \,\ \vert x \vert \geq \varepsilon.
\end{array}\right.
\]

According to Chen \cite[Lemma 3.4]{Chen}, we have that, for $\varepsilon >0$ small, $\phi_{\varepsilon} \in H_{\varepsilon}^{s}(\Omega)$ and the following estimate is valid
\begin{equation}\label{Est1}
\Vert \phi_{\varepsilon} \Vert_{H_{\varepsilon}^{2}(\Omega)}^{2} \leq \frac{C}{\varepsilon^{N}},\qquad \textrm{where}\ \ C= C(N,s, \Omega).
\end{equation}

Moreover, we have (see \cite[Equation 2.11]{MR929196})
\begin{equation}\label{Des2}
\int_{\Omega} \vert \phi_{\varepsilon}(x) \vert^{q} \dd x = K_{q}\varepsilon^{(1-q)N}, \,\,\,\, \text{with} \,\,\,\, K_{q} = N\Omega_{N} \int_{0}^{1} (1-\rho)^{q}\rho^{N-1} \dd \rho.
\end{equation}
\vspace{0.3cm}

The following lemmas are adaptations of results in  Lin, Ni e Takagi \cite{MR929196}.
\begin{lemma}\label{Sig}
There exists a unique $\sigma \in (0,1)$ such that
\[
\int_{\Omega_{\sigma}} \vert \phi_{\varepsilon}(x) \vert^{2} \dd x = \frac{1}{2} \int_{\Omega} \vert \phi_{\varepsilon}(x) \vert^{2} \dd x
\]	
where $\Omega_{\sigma} = \{x \in \Omega\,:\, \phi_{\varepsilon}(x) > \sigma \varepsilon^{-N}\}$.
\end{lemma}
\begin{proof}
In fact, note that if $\sigma \in (0,1)$ and $\phi_{\varepsilon}(x) > \sigma \varepsilon^{-N}$, then $\vert x \vert < (1-\sigma)\varepsilon$. Thus
\begin{align*}
\int_{\Omega_{\sigma}} \vert \phi_{\varepsilon}(x) \vert^{2} dx &= \int_{B_{(1-\sigma)\varepsilon}(0)}\varepsilon^{-2N}\left(1 - \displaystyle\frac{\vert x \vert}{\varepsilon}\right)^{2} \dd x = \frac{1}{\varepsilon^{2N+2}} \int_{B_{(1-\sigma)\varepsilon}(0)}\left(\varepsilon - \vert x \vert \right)^{2} \dd x \\
&= \frac{N\Omega_{N}}{\varepsilon^{2N + 2}} \int_{0}^{(1-\sigma)\varepsilon}\left(\varepsilon - r\right)^{2}r^{N-1} \dd r\\
&= \frac{N\Omega_{N}}{\varepsilon^{2N+2}} \int_{0}^{(1-\sigma)\varepsilon} \left[\varepsilon^{2}r^{N-1} - 2\varepsilon  r^{N} + r^{N+1}\right] \dd r\\
&= \frac{N\Omega_N(1-\sigma)^{N}}{\varepsilon^{N}}\left[\displaystyle\frac{1}{N} - \displaystyle\frac{2(1-\sigma)}{N+1} + \displaystyle\frac{(1-\sigma)^{2}}{N+2}\right].
\end{align*}

On the other hand, taking $q=2$ in \eqref{Des2}, we obtain
\[
\int_{\Omega} \vert \phi_{\varepsilon} \vert^{2} \dd x = \frac{N\Omega_{N}}{\varepsilon^{N}} \left[\displaystyle\frac{1}{N} - \displaystyle\frac{2}{N+1} + \displaystyle\frac{1}{N+2}\right].
\]

Thus, we conclude the claim just by taking $\sigma \in (0,1)$ such that
\[
(1-\sigma)^{N} \left[\displaystyle\frac{1}{N} - \displaystyle\frac{2(1-\sigma)}{N+1} + \displaystyle\frac{(1-\sigma)^{2}}{N+2}\right] = \left[\displaystyle\frac{1}{N} - \displaystyle\frac{2}{N+1} + \displaystyle\frac{1}{N+2}\right].
\vspace*{-.5cm}\]

$\hfill\Box$\end{proof}\vspace*{.3cm}

Now, consider the function $g\colon [0,\infty) \to \mathbb{R}$ defined by
\begin{equation}\label{g}
g(t) = I_{\varepsilon}(t\phi_{\varepsilon})= \frac{t^{2}}{2} \Vert \phi_{\varepsilon} \Vert_{H_{\varepsilon}^{2}(\Omega)}^{2} - \int_{\Omega}F(t\phi_{\varepsilon}) \dd x.
\end{equation}
\begin{lemma}\label{PG}
There exist $t_1,t_2\in [0,\infty)$ with $0<t_1<t_2$ such that
\begin{enumerate}
\item [$(i)$] $g'(t) < 0$ if $t> t_1$;
\item [$(ii)$] $g(t) < 0$ if $t\geq t_2$.
\end{enumerate}
\end{lemma}
\begin{proof}
Taking the derivative in \eqref{g} and applying estimate \eqref{Est1}, we obtain
\begin{align}\label{g1}
g'(t)= t\Vert \phi_{\varepsilon} \Vert_{H_{\varepsilon}^{2}(\Omega)}^{2} - \int_{\Omega}f(t\phi_{\varepsilon})\phi_{\varepsilon}\dd x
\leq \frac{tC}{\varepsilon^{N}} - \int_{\Omega}f(t\phi_{\varepsilon})\phi_{\varepsilon}\dd x.
\end{align}

Note that condition $(f_3)$  implies that, for any $R>0$, there exists $M_{R} > 0$ such that for all $\xi \geq M_{R}$, we have
\begin{equation}\label{Des3}
f(\xi)\geq R\xi. 
\end{equation}

Denote
\[
\Omega_{1} = \left\{x \in \Omega\,:\, \phi_{\varepsilon}(x) > \frac{M_{R}}{t}\right\}.
\]

Keeping in mind Lemma \ref{Sig}, note that $\Omega_{\sigma} \subset \Omega_{1}$ for $t > \frac{M_{R}\varepsilon^{N}}{\sigma}$. Since $f(t)>0$, for such $t$, it follows from \eqref{Des3} that
\begin{align*}
\int_{\Omega} f(t\phi_{\varepsilon}(x))\phi_{\varepsilon}(x) \dd x &\geq \int_{\Omega_{1}} f(t\phi_{\varepsilon}(x))\phi_{\varepsilon}(x) \dd x\geq \int_{\Omega_1} Rt\phi_{\varepsilon}(x)\phi_{\varepsilon}(x) \dd x\\
&\geq Rt \int_{\Omega_{\sigma}} \left(\phi_{\varepsilon}(x) \right)^{2} \dd x.
\end{align*}

Substituting into \eqref{g1} and applying Lemma \ref{Sig}, we obtain 
\begin{align*}
g'(t) &\leq \frac{Ct}{\varepsilon^{N}} - Rt\int_{\Omega_{\sigma}} \left(\phi_{\varepsilon}(x) \right)^{2} \dd x= \frac{Ct}{\varepsilon^{N}} - \frac{Rt}{2}\int_{\textcolor{blue}{\Omega}} \left(\phi_{\varepsilon}(x) \right)^{2} \dd x=t\varepsilon^{-N}\left(C-\frac{K_2 R}{2}\right).
\end{align*}
where $K_2$ was defined in \eqref{Des2}. So, for $R_1 > \frac{2C}{K_2}$, we have
\[
g'(t) < 0 \qquad \text{for any} \qquad t> \frac{M_{R}\varepsilon^{N}}{\sigma}= t_1.
\]

In order to prove $(ii)$, note that $(f_1)$ and \eqref{Des3} imply that, for any $\xi \geq M_{R}$, we have
\begin{align*}
F(\xi) = \int_{0}^{\xi} f(\tau)\dd \tau &= \int_{0}^{M_{R}}f(\tau) \dd \tau + \int_{M_{R}}^{\xi}f(\tau) \dd \tau \\
&\geq \int_{M_{R}}^{\xi} R\tau \dd \tau = \frac{R\xi^{2}}{2} - m_{R} 
\end{align*}
where $m_{R} = \displaystyle\frac{M_{R}^{2} R}{2}$.

Applying again \eqref{Des2}, we obtain 
\begin{align}\label{g2}
g(t) &= \frac{t^2}{2} \Vert \phi_{\varepsilon} \Vert_{H_{\varepsilon}^{s}}^{2} - \int_{\Omega}F(t\phi_{\varepsilon}) \dd x\\
&\leq \frac{t^2 C}{2\varepsilon^{N}} - \frac{RK_2t^{2}}{2 \varepsilon^{N}} + m_R \vert \Omega \vert \nonumber\\
&= \frac{t^2}{2\varepsilon^{N}}\left(C - RK_2\right) + m_R\vert \Omega \vert.\nonumber 
\end{align}

Taking $R_2 > \frac{C}{K_2}$, it follows that
\[g(t)<0\quad\textrm{for all}\quad t>0\quad\textrm{such that}\quad t^2>\frac{2m_r|\Omega|\varepsilon^N}{R_2K_2-C}.\]\goodbreak

In order to have $t_2 > t_1$, we take $t_2$ satisfying 
\[
t_2>\frac{M_{R}\varepsilon^{N}}{\sigma} \qquad \text{and}\qquad t_{2}^2 > \displaystyle\frac{2m_R \vert \Omega \vert \varepsilon^{N}}{R_2K_2 - C}.
\]
We are done.
$\hfill\Box$\end{proof}\vspace*{.3cm}

\begin{lemma}\label{Lim}
For all $\varepsilon > 0$ sufficiently small, there exists a nonnegative function $\phi \in H_{\varepsilon}^{s}(\Omega)$ and $t_0 >0$ such that $I_{\varepsilon}(t_0 \phi) = 0$. Moreover, there is $C= C(N,s,\Omega)>0$ 
\[
I_{\varepsilon}(t\phi) \leq C \varepsilon^{N}\quad\textrm{for all }\ t. 
\]
\end{lemma}
\begin{proof}
According to Lemma \ref{gpm}, we have $g(t) >0$ for $t$ sufficiently small. Lemma \ref{PG} and $(f_1)$ imply that $g(t)\geq 0$ for $0<t<t_1$. Thus, by substituting \eqref{Est1} into \eqref{g2}, we obtain
\begin{align*}
\max_{t\geq 0}g(t)= \max_{0\leq t \leq t_{1}} g(t)&\leq \max_{0\leq t \leq t_{1}}\left\{\frac{Ct^{2}}{2\varepsilon^{N}} - \int_{\Omega}F(t\phi_{\varepsilon}) \dd x \right\}\leq \max_{0\leq t \leq t_{1}} \frac{Ct^2}{2\varepsilon^{N}} = \frac{Ct_{1}^{2}}{2\varepsilon^{N}}.
\end{align*}
Since $t_{1} = \frac{M_{R}\varepsilon^{N}}{\sigma}$, we have 
\[
I_{\varepsilon}(t\phi_{\varepsilon}) = g(t) \leq \max_{t\geq 0}g(t) =\frac{CM^2_R\varepsilon^{2N}}{2\varepsilon^N\sigma^2}= C_1\varepsilon^{N}\] 
for a positive constant $C_1$. The existence of $t_0>t_1$ also follows from Lemma \ref{PG}.
$\hfill\Box$\end{proof}\vspace*{.3cm} 
 
\noindent\textbf{Theorem 1.}\textit{
 Assume $(f_1)$-$(f_5)$. Then, for $\varepsilon$ is sufficiently small, there exists a non-constant, nonnegative solution of \eqref{P} satisfying 
 \[
 I_{\varepsilon}(u_\varepsilon) \leq C\varepsilon^{N}
 \]  
 where $C>0$ depends only on $\Omega$ and $f$.}\vspace*{.2cm}

\textit{Proof.}  Choose $t_2$ as in Lemma \ref{PG} and define $e=t_2\varphi_\varepsilon\in H_{\varepsilon}^{s}$. The geometry of the Mountain Pass Theorem was obtained in Lemmas \ref{gpm} and \ref{PG}, while the (PS)-condition was proved in Lemma \ref{PScond}. Considering 
\[
\Gamma = \left\{ \gamma \in C([0,1];H_{\varepsilon}^{s}(\Omega)) ; \,\,\, \gamma(0) = 0 \,\,\, \text{and} \,\,\, \gamma(1)=e \,\, \right\},
\]
the value
\[
c_{\varepsilon}:= \inf_{\gamma \in \Gamma} \max_{0\leq t \leq 1} I_{\varepsilon}(\gamma(t)) \geq \delta > 0
\]
is a critical value of $I_{\varepsilon}$. Therefore, there exists $u_{\varepsilon} \in H_{\varepsilon}^{s}(\Omega)$ such that,
\[
I_{\varepsilon}(u_{\varepsilon}) = c_{\varepsilon} \,\,\, \text{and} \,\,\, I_{\varepsilon}'(u_\varepsilon) = 0.
\]

In particular,  Lemma \ref{Lim} implies that
\[
I_{\varepsilon}(u_{\varepsilon}) = c_{\varepsilon} \leq \max_{0\leq t \leq t_2} I_{\varepsilon}(t\phi_{\varepsilon}) \leq C \varepsilon^{N}.
\]

Observe that, if $u = \mu$ is a solution to our problem \eqref{P}, then
\[
f(\mu) = \mu.
\]

It follows from condition $(f_5)$ that
\begin{align*}
I_{\varepsilon}(\mu) &= \frac{\mu^{2}}{2} - \int_{\Omega} F(\mu) \dd x=\left(\frac{\mu^{2}}{2} - F(\mu) \right) \vert \Omega \vert \geq \alpha \vert \Omega \vert > 0 .
\end{align*}
Thus, for $\varepsilon < \left(\frac{\alpha \vert \Omega \vert}{C}\right)^{\frac{1}{N}}$ we obtain
\[
I_{\varepsilon}(u_{\varepsilon}) \leq C \varepsilon^{N} < \alpha \vert \Omega \vert = I_{\varepsilon}(\mu),
\]
meaning that, for $\varepsilon>0$ sufficiently small, $u_{\varepsilon}$ can not be constant, and therefore, is nontrivial.

Finally, condition $(f_1)$ implies that $f(u_{\varepsilon}) = 0$ if $x \in  \{ x \in \Omega; u_{\varepsilon} \leq 0\}$. Thus, denoting for $u_{\varepsilon}^{-}= \max\{-u_{\varepsilon},0\}$ we have
\[
\int_{\Omega} f(u_{\varepsilon})\,u_{\varepsilon}^{-} \dd x = \int_{\{x \in \Omega; \,\, u_{\varepsilon} >0\}} f(u_{\varepsilon})\,u_{\varepsilon}^{-} \dd x  +  \int_{\{x \in \Omega; \,\, u_{\varepsilon} \leq 0\}} f(u_{\varepsilon})\,u_{\varepsilon}^{-} \dd x = 0.
\]

Therefore,
\begin{align*}
0 &=I'_{\varepsilon}(u_{\varepsilon})\cdot u_{\varepsilon}^{-} \\
&=\frac{C_{N,s} \varepsilon^{2s}}{2} \iint_{\mathbb{R}^{2N} \backslash (\Omega^{c})^{2}} \frac{(u_{\varepsilon}(x) -u_{\varepsilon}(y))(u_{\varepsilon}^{-}(x) -u_{\varepsilon}^{-}(y))}{\vert x-y \vert^{N+2s}} \dd x \dd y + \int_{\Omega} \vert u_{\varepsilon}^{-} \vert^{2}\dd x.
\end{align*}

Now, the inequality $(\xi - \eta)(\xi^{-} - \eta^{-}) \geq \vert \xi^{-} - \eta^{-}\vert^{2}$ guarantees that
\[
\frac{C_{N,s} \varepsilon^{2s}}{2} \iint_{\mathbb{R}^{2N} \backslash (\Omega^{c})^{2}} \frac{\vert u_{\varepsilon}^{-}(x) -u_{\varepsilon}^{-}(y)\vert^{2}}{\vert x-y \vert^{N+2s}} \dd x \dd y + \int_{\Omega} \vert u_{\varepsilon}^{-}\vert^{2} \dd x = 0,
\]
proving that $u_{\varepsilon}^{-} \equiv 0$, that is, $u_{\varepsilon} \geq 0$.
$\hfill\Box$
 
\begin{corollary}
Assume conditions $(f_1)$-$(f_5)$ with $a_3 = 0$. If $u_\varepsilon$ is a solution of \eqref{P}, then there exists a constant $K_0 > 0$ such that
\[
\Vert u_{\varepsilon} \Vert_{H_{\varepsilon}^{s}}^{2} = \int_{\Omega} f(u_{\varepsilon})\,u_{\varepsilon} \dd x \leq  K_0 \varepsilon^{N}.
\]
\end{corollary}
\begin{proof} Since $I'_\varepsilon(u_\varepsilon)\cdot u_\varepsilon=0$, we have
\[
\frac{C_{N,s} \varepsilon^{2s}}{2} \iint_{\mathbb{R}^{2N} \backslash (\Omega^{c})^{2}} \frac{\vert u_{\varepsilon}(x) - u_{\varepsilon}(y) \vert^{2}}{\vert x-y\vert^{N+2s}} \dd x \dd y + \int_{\Omega} \vert u_{\varepsilon} \vert^{2} \dd x = \int_{\Omega} f(u_{\varepsilon})\,u_{\varepsilon} \dd x,
\]	
that is,
\[\Vert u_{\varepsilon} \Vert_{H_{\varepsilon}^{2}(\Omega)}^{2}=\int_{\Omega} f(u_{\varepsilon})\,u_{\varepsilon} \dd x.\]
Theorem \ref{Exis} and $(f_4)$ yield
\begin{align*}
C \varepsilon^{N} &\geq I_{\varepsilon}(u_{\varepsilon}) = \frac{1}{2} \Vert u_{\varepsilon} \Vert_{H_{\varepsilon}^{2}(\Omega)}^{2} - \int_\Omega F(u_{\varepsilon}) \dd x \\
&\geq \frac{1}{2} \Vert u_{\varepsilon} \Vert_{H_{\varepsilon}^{2}(\Omega)}^{2} - \frac{1}{\theta} \int_\Omega f(u_{\varepsilon})\,u_{\varepsilon} \dd x= \left(\frac{1}{2} - \frac{1}{\theta} \right) \Vert u_{\varepsilon} \Vert_{H_{\varepsilon}^{2}(\Omega)}^{2}\\
\end{align*}

Since $\theta>2$, we obtain that
\[
\Vert u_{\varepsilon} \Vert_{H_{\varepsilon}^{2}(\Omega)}^{2} \leq K_0 \varepsilon^{N}.
\]

\vspace*{-.4cm}$\hfill\Box$\end{proof}
 
\section{Proof of Theorem \ref{Limita}}

\noindent\textbf{Theorem 2.}\textit{
Suppose $0<s<1$, $(f_1)$-$(f_3)$ holds. If $u_\varepsilon$ is a solution  to problem \eqref{P} with $\varepsilon > 0$ small enough, then $u_{\varepsilon} \in  L^{\infty}(\Omega)$.}\vspace*{.3cm}

\textit{Proof.} In order to simplify the notation, we denote $u = u_{\varepsilon}$ a solution of \eqref{P} for $\varepsilon >0$ sufficiently small. Theorem \ref{Exis} guarantees that $u \geq 0$. Given $\alpha > 1$ and $M>0$, consider the functions $u_{M}= \min\left\{u, M \right\}$ and
\[
g_{\alpha, M}(t) = t\left( \min\{t, M\} \right)^{\alpha -1}=\left\{\begin{array}{rc}
t^{\alpha},&\mbox{if}\quad t\leq M\\
tM^{\alpha -1}, &\mbox{if}\quad t>M.
\end{array}\right.
\]

Since $g_{\alpha, M}$ is Lipschitz continuous and increasing, we conclude that $g_{\alpha, M}(u) \in H_{\varepsilon}^s(\Omega)$, for all $u \in H_{\varepsilon}^{s}(\Omega)$.

Thus,
\begin{equation*}
I_{\varepsilon}'(u)\cdot g_{\alpha,M}(u) = 0
\end{equation*}
that is,
\begin{align}\label{SF}
\frac{C_{N,s} \varepsilon^{2s}}{2}  \iint_{\mathbb{R}^{2N} \backslash (\Omega^{c})^{2}} &\frac{\left(u(x) - u(y) \right)\left(g_{\alpha,M}(u)(x) - g_{\alpha,M}(u)(y)\right)}{\vert x-y\vert^{N+2s}} \dd x \dd y + \int_{\Omega} u\,g_{\alpha,M}(u) \dd x\nonumber \\
& = \int_{\Omega} f(u)\,g_{\alpha,M}(u) \dd x.
\end{align}

We define the function,
\[
G_{\alpha,M}(t) = \int_{0}^{t} \left(g_{\alpha,M}'(\tau)\right)^{\frac{1}{2}} \dd \tau.
\]
A direct calculation shows that
\begin{equation}\label{Des6}
G_{\alpha,M}(t) \geq \frac{2}{\alpha + 1} t \left(\min\{t,M\}\right)^{\frac{\alpha-1}{2}}, \quad \text{for all} \quad t \in \mathbb{R}.
\end{equation}

Moreover,
\begin{equation}\label{Des7}
\vert G_{\alpha,M}(a) -G_{\alpha,M}(b) \vert^{2} \leq \left(g_{\alpha,M}(a) - g_{\alpha,M}(b)\right)\left(a-b\right) , \,\,\, \forall a,b \in \mathbb{R}.
\end{equation}
It follows from \eqref{Des7} 
that
\begin{align*}
[G_{\alpha,M}(u)]_{s,2}^{2} &:= \frac{C_{N,s}\varepsilon^{2s}}{2}\int_{\Omega} \int_{\Omega} \frac{\vert G_{\alpha,M}(u(x))-G_{\alpha,M}(u(y))\vert^{2}}{\vert x-y \vert^{N+2s}} \dd x \dd y \\
&\leq \frac{C_{N,s}\varepsilon^{2s}}{2}\iint_{\mathbb{R}^{2N} \backslash \left(\Omega^{c}\right)^{2}} \frac{\left(u(x)-u(y)\right)\left( g_{\alpha,M}(u(x))-g_{\alpha,M}(u(y))\right)}{\vert x-y \vert^{N+2s}} \dd x \dd y
\end{align*}
Lemma \ref{DS} (i.e., the Sobolev inequality) and \eqref{SF} yield
\begin{align*}
S^{-2} \varepsilon^{2s}\left(\int_{\Omega} \vert G_{\alpha,M}(u) \vert^{2_{s}^{*}} \dd x\right)^{\frac{2}{2_{s}^{*}}} &\leq  \left[G_{\alpha,M}(u)\right]_{s,2}^{2} + \int_{\Omega}\vert G_{\alpha,M}(u) \vert^{2} \dd x \\
& \leq \int_{\Omega} f(u)\,g_{\alpha,M}(u) \dd x= \int_{\Omega} f(u)\,u\,u_{M}^{\alpha-1} \dd x.
\end{align*}

Combining \eqref{Des6} with the last inequality, we obtain
\begin{align}\label{Des8}
S^{-2} \varepsilon^{2s}\left[\left(\frac{2}{\alpha + 1}\right)^{2_{s}^{*}}  \int_{\Omega} \left\vert u\,u_{M}^{\frac{\alpha-1}{2}} \right\vert^{2_{s}^{*}} \dd x \right]^{\frac{2}{2^*_s}}&\leq S^{-2} \varepsilon^{2s}\left(\int_{\Omega} \left\vert G_{\alpha,M}(u) \right\vert^{2_{s}^{*}} \dd x\right)^{\frac{2}{2_{s}^{*}}}\nonumber \\
&\leq \int_{\Omega} f(u)\,u\,u_{M}^{\alpha-1} \dd x.
\end{align}

According to Remark \ref{obs1}, for any fixed $C_\eta>0$, there exists $\eta>0$ such that 
\begin{equation}\label{Des9}
\vert f(t) \vert \leq  \eta\vert t \vert +C_\eta\vert t \vert^{2_{s}^{*}-1} + , \quad \forall t \in \mathbb{R}.
\end{equation}

Applying \eqref{Des9} and the Hölder inequality, we estimate the right-hand of \eqref{Des8}. 
\begin{align*}
\int_{\Omega} f(u)\,u\,u_{M}^{\alpha-1} \dd x &\leq \eta \int_{\Omega}u^{2}\,u_{M}^{\alpha - 1} \dd x+C_\eta \int_{\Omega} u^{2_{s}^{*}}\,u_{M}^{\alpha - 1} \dd x  \\
&=\eta \int_{\Omega}u^{2}\,u_{M}^{\alpha - 1} \dd x +C_\eta \int_{\Omega} u^{2_{s}^{*}-2}\,u^{2}\,u_{M}^{\alpha - 1} \dd x \\
&= \eta \int_{\Omega}u^{2}\,u_{M}^{\alpha - 1} \dd x+C_\eta \int_{\Omega} u^{\frac{4s}{N-2s}} \left\vert u\,u_{M}^{\frac{\alpha - 1}{2}}\right\vert^{2} \dd x\\
&\leq \eta \int_{\Omega}u^{2}\,u_{M}^{\alpha - 1} \dd x+C_\eta \left\Vert u \right\Vert_{2_{s}^{*}}^{\frac{4s}{N-2s}} \left(\int_{\Omega} \left\vert u\,u_{M}^{\frac{\alpha - 1}{2}} \right\vert^{2_{s}^{*}} \dd x\right)^{\frac{N-2s}{N}}.
\end{align*}
Choosing $C_\eta >0$ small enough such that
\[
C_\eta \left\Vert u \right\Vert_{2_{s}^{*}}^{\frac{4s}{N-2s}} \leq \frac{S^{-2}\varepsilon^{2s}}{2} \left(\frac{2}{\alpha +1 }\right)^{2},
\]
we conclude that
\[\int_\Omega f(u)uu^{\alpha-1}_M\dd x\leq \eta \int_\Omega u^2u^{\alpha-1}_M+\frac{S^{-2}e^{2s}}{2}\left(\frac{2}{\alpha+1}\right)\left(\int_\Omega \left|uu^{\frac{\alpha-1}{2}}\right|^{2^*_s}\right)^{\frac{2}{2^*_s}}.\]	
Combining with \eqref{Des8} yields
\[
\frac{S^{-2}\varepsilon^{2s}}{2}\left(\frac{2}{\alpha +1 }\right)^{2}\left(\int_{\Omega} \left\vert u\,u_{M}^{\frac{\alpha - 1}{2}} \right\vert^{2_{s}^{*}} \dd x\right)^{\frac{2}{2_{s}^{*}}} \leq
\eta \int_{\Omega}u^{2}\,u_{M}^{\alpha - 1} \dd x.
\]

Thus, for a positive constant $C$,
\begin{align*}
\left(\int_{\Omega} \left\vert u\,u_{M}^{\frac{\alpha - 1}{2}} \right\vert^{2_{s}^{*}} \dd x\right)^{\frac{2}{2_{s}^{*}}} \leq C(\alpha + 1)^{2}  \int_{\Omega}u^{2}\,u_{M}^{\alpha - 1} \dd x.
\end{align*}

Making $M \to \infty$, Fatou's lemma and the dominate convergence theorem yield
\[
\left\Vert u \right\Vert_{2_{s}^{*}\left(\frac{\alpha+1}{2}\right)}^{\alpha +1} \leq C(\alpha +1 )^{2} \left\Vert u \right\Vert_{2\left(\frac{\alpha + 1}{2}\right)}^{\alpha + 1}
\]
and taking $\beta = \frac{\alpha + 1}{2}$, we obtain
\begin{equation}\label{Des10}
\left\Vert u \right\Vert_{2_{s}^{*}\beta}^{2 \beta} \leq C\beta^{2} \left\Vert u \right\Vert_{2 \beta}^{2\beta}.
\end{equation}

Now, choose $K > 1$ such that $C^{\frac{1}{2}}\beta \leq K e^{\sqrt{\beta}}$.
Then, \eqref{Des10} can be written as
$\left\Vert u \right\Vert_{2_{s}^{*} \beta}^{\beta} \leq Ke^{\sqrt{\beta}} \left\Vert u \right\Vert_{2\beta}^{\beta}$.

Thus,
\begin{align}\label{Des11}
\left\Vert u \right\Vert_{2_{s}^{*} \beta}\leq K^{\frac{1}{\beta}}e^{\frac{1}{\sqrt{\beta}}} \left\Vert u \right\Vert_{2\beta}, \quad \text{for all} \quad \beta > 0. 
\end{align}

Consider the sequence defined by
\[
\beta_{1} = 1, \quad \beta_{n+1} = \left(\frac{2_{s}^{*}}{2}\right)\beta_{n} \quad \text{for} \quad n=\mathbb{N}=\{1,2,\ldots\}.
\]

Since $\frac{\beta_{n}}{\beta_{n+1}} =\frac{2}{2^*_s}<1$, the series 
\begin{equation}\label{Conv}
\sum_{n=0}^{\infty} \frac{1}{\beta_{n}} \quad \text{and} \quad \sum_{n=0}^{\infty} \frac{1}{\sqrt{\beta_{n}}}
\end{equation}
are both convergent.

Using the sequence $(\beta_{n})$ in \eqref{Des11} and iterating we obtain
\[\|u\|_{2^*_s}\beta_2\leq K^{\frac{1}{\beta_2}}e^{\frac{1}{\sqrt{\beta_2}}}\|u\|_{2^*_s}=K^{\frac{1}{1}+\frac{1}{\beta_2}}e^{\frac{1}{\sqrt{1}}+\frac{1}{\sqrt{\beta_2}}}\|u\|_2.\]
Proceeding repeatedly, yields
\[\left\Vert u \right\Vert_{2_{s}^{*} \beta_{n}}\leq K^{\left(\displaystyle\sum_{i=0}^{n} \frac{1}{\beta_{i}}\right)} e^{\left(\displaystyle\sum_{i=0}^{n} \frac{1}{\sqrt{\beta_{i}}}\right)} \left\Vert u \right\Vert_{2}.\]

Making $n\to\infty$, we conclude that, for positive constants $\gamma_1$ and $\gamma_{2}$
\[\left\Vert u \right\Vert_{\infty} \leq K^{\gamma_1}e^{\gamma_{2}} \left\Vert u \right\Vert_{2} < \infty,
\]
%
%
that is, $u \in L^{\infty}(\Omega)$. 
$\hfill\Box$

\noindent \textbf{Acknowledgements:} The authors thank Gilberto A. Pereira for useful conversations. 
\bibliographystyle{plain}
\bibliography{biblio_jef}
\end{document}